\documentclass{article}
\usepackage{amsmath}
\usepackage{amsfonts}
\usepackage{latexsym,amsthm,amscd}
\newtheorem{thm}{Theorem}
\newtheorem{cor}{Corollary}
\newtheorem{prop}{Proposition}

\newtheorem{lem}{Lemma}
\newtheorem{Def}{Definition}

\newcounter{alphthm}
\newcommand{\be}{\begin{equation}}
\newcommand{\ee}{\end{equation}}
\newcommand{\ben}{\begin{enumerate}}
\newcommand{\een}{\end{enumerate}}
\newcommand{\pa}{{\partial}}

\newcommand{\bc}{\begin{cor}}
\newcommand{\ec}{\end{cor}}
\newcommand{\beqn}{\begin{eqnarray*}}
\newcommand{\eeqn}{\end{eqnarray*}}
\newcommand{\bpf}{\begin{proof}}
\newcommand{\epf}{\end{proof}}
\newcommand{\bl}{\begin{lem}}
\newcommand{\el}{\end{lem}}
\newcommand{\bp}{\begin{prop}}
\newcommand{\ep}{\end{prop}}
\newcommand{\bd}{\begin{Def}}
\newcommand{\ed}{\end{Def}}
\newcommand{\bt}{\begin{thm}}
\newcommand{\et}{\end{thm}}
\title{\Large On Generalized $m$-th Root Finsler Metrics }
\author{A. Tayebi, E. Peyghan and M. Shahbazi}
\begin{document}
\maketitle

\begin{abstract}
In this paper, we  characterize locally dually flat generalized $m$-th root Finsler metrics. Then we find a condition under which a generalized $m$-th root metric is projectively related to a $m$-th root metric. Finally, we  prove that if a generalized  $m$-th root metric  is conformal to a  $m$-th root  metric, then  both of them reduce to Riemannian metrics.
\end{abstract}

\section{Introduction}

An $m$-th root  metric $F=\sqrt[m]{A}$, where $A := a_{i_{1}}..._{i_{m}}(x)y^{i_1} . . . y^{i_m}$,  is regarded as a direct generalization of Riemannian
metric in a sense, i.e., the second root metric is a Riemannian metric.  The theory of $m$-th root metrics has been developed by Matsumoto-Shimada \cite{MatShim}\cite{Shim}, and applied by Antonelli to Biology as an ecological metric \cite{AIM}. The third and fourth root metrics are called the cubic metric and quartic metric, respectively.

For quartic metrics, a study of the geodesics and of the related geometrical objects is made by Balan, Brinzei and Lebedev  \cite{B}\cite{BBL}\cite{Leb}. Also, Einstein equations for some relativistic models relying on such metrics are studied by  Balan-Brinzei in two papers \cite{Mangalia}\cite{Cairo1}. In four-dimension, the special quartic metric in the form $F=\sqrt[4]{y^1y^2y^3y^4 }$ is called the Berwald-Mo\'{o}r metric \cite{Balan}\cite{Balan2}.  In the last two decades, physical studies due to Asanov, Pavlov and their co-workers emphasize the important role played by the Berwald-Mo\'{o}r metric in the theory of space-time structure and gravitation as well as in unified gauge field theories \cite{Asan}\cite{Pav1}\cite{Pav2}. In \cite{Balan}, Balan prove that the Berwald-Mo\'{o}r structures are pseudo-Finsler of Lorentz type and  for co-isotropic submanifolds of Berwald-Mo\'{o}r spaces present the Gauss-Weingarten, Gauss-Codazzi, Peterson-Mainardi and Ricci-K\"{u}hne equations.

In  \cite{Tamassy}, tensorial connections for $m$-th root Finsler metrics have been studied by Tamassy. Li-Shen study locally projectively flat fourth root metrics under irreducibility condition \cite{shLi1}. Yu-You show that an $m$-th root Einstein Finsler metrics are Ricci-flat \cite{YY}. In  \cite{TN1}, Tayebi-Najafi characterize locally dually flat and Antonelli $m$-th root metrics. They prove that every  $m$-th root metric of isotropic mean Berwald curvature (resp, isotopic Lanbdsberg curvature)   reduces to  a weakly Berwald metric (resp, Landsberg metric). They show that $m$-th root metric with almost vanishing $H$-curvature has vanishing $H$-curvature \cite{TN2}.

Let $(M, F)$ be a Finsler manifold of dimension $n$, $TM$ its tangent bundle and $(x^i, y^i)$ the
coordinates in a local chart on TM. Let $F$ be a scalar function on $TM$ defined by $F=\sqrt{A^{2/m}+B}$,  where $A$ and $B$ are given by
\'{A}\be
A := a_{i_{1}}..._{i_{m}}(x)y^{i_1} . . . y^{i_m}, \ \ \    B :=b_{ij}(x)y^iy^j.
\ee
Then $F$ is called generalized  $m$-th root Finsler metric. Put
\begin{eqnarray*}
A_{i}={\pa A\over \pa y^i}, \ \  A_{ij}={\pa^2 A\over \pa y^j\pa y^j}, \ \   B_{i}={\pa B\over \pa y^i}, \ \  B_{ij}={\pa^2 B\over \pa y^j\pa y^j},\\
A_{x^i}=\frac{\partial A}{\partial x^i}, \ \  A_0=A_{x^i}y^i, \ \  B_{x^i}=\frac{\partial B}{\partial x^i}, \ \  B_0=B_{x^i}y^i.
%\\ A_{0l}=A_{x^k y^l}y^k=\frac{\partial^2 A}{\partial x^i \partial y^l}y^k, \ \  \ B_{0l}=B_{x^k y^l}y^k=\frac{\partial^2 B}{\partial x^i \partial y^l}y^k.
\end{eqnarray*}
%Then\[B_{i}=2b_{im}y^m, \ \  B_{ij}=b_{ij}.\]
Suppose that the matrix $(A_{ij})$ defines a positive definite tensor and $(A^{ij})$  denotes its inverse. Then the following hold
\begin{eqnarray}
&&g_{ij}=\frac{A^{\frac{2}{m}-2}}{m^2}[mAA_{ij}+(2-m)A_iA_j]+b_{ij},\label{gg}\\
&&y^iA_i=mA, \ \ y^iA_{ij}=(m-1)A_j,\ \ y_i=\frac{1}{m}A^{\frac{2}{m}-1}A_i,\\
&&A^{ij}A_{jk}=\delta^i_k,\ \ A^{ij}A_i=\frac{1}{m-1}y^j, \ \ A_iA_jA^{ij}=\frac{m}{m-1}A.
\end{eqnarray}

Information geometry has emerged from investigating the geometrical structure of a family of probability
distributions and has been applied successfully to various areas including statistical inference, control system theory and multi-terminal information theory \cite{am}\cite{amna}. Dually flat Finsler metrics form a special and valuable class of Finsler metrics in Finsler information geometry, which plays a very important role in studying flat Finsler information structure \cite{shen1}. A Finsler metric $F$ on a manifold $M$ is said to be locally dually flat, if at any point there is a standard coordinate system $(x^i, y^i)$ in $TM$ such that $(F^2)_{x^{k}y^{l}}y^k = 2(F^2)_{x^{l}}$. In this case, the coordinate $(x^i)$ is called an adapted local coordinate system. In this paper, we characterize locally dually flat  generalized $m$-th root Finsler metrics. More precisely, we prove the following.
\begin{thm}\label{mainthm1}
Let $F=\sqrt{A^{2/m}+B}$ be a generalized  $m$-th root metric on an open  subset $U\subset \mathbb{R}^n$. Suppose that $A$ is irreducible. Then $F$ is locally dually flat if and only if there exists a 1-form $\theta = \theta_{l} (x)y^l$  on U such that
the following holds
\begin{eqnarray}
&&B_{0l}=2B_{x^{l}},\label{SRR1}\\
&&A_{x^l}=\frac{1}{3m}[mA\theta_{l}+2\theta A_{l}],\label{SRR2}
\end{eqnarray}
where $B_{0l}=B_{x^k y^l}y^k$.
\end{thm}

In  local coordinates $(x^i,y^i)$, the vector filed ${\bf G}=y^i\frac{\pa}{\pa x^i}-2G^i\frac{\pa}{\pa y^i}$ is a global vector field on $TM_0$, where $G^i=G^i(x,y)$ are local functions on $TM_0$  given by following
\[
G^i:=\frac{1}{4}g^{il}\Big\{ \frac{\partial^2F^2}{\partial x^k \partial y^l}y^k-\frac{\partial F^2}{\partial x^l}\Big\},\ \ y\in T_xM.
\]
The vector field ${\bf G}$ is called the associated spray to $(M,F)$. Two Finsler metrics $F$ and  ${\bar F}$ on a  manifold $M$ are called projectively related if  there is a scalar function $P(x, y)$ defined on $TM_0$ such that
\[
{\bar G}^i= G^i+Py^i,
\]
where   ${\bar G}^i$ and $G^i$ are the geodesic spray coefficients of ${\bar F}$ and $F$,  respectively.

\begin{thm}\label{mainthm2}
Let ${\bar F}=\sqrt{A^{2/m}+B}$ and $F=A^{1/m}$ are  generalized  $m$-th root and  $m$-th root Finsler metrics on an open  subset $U\subset \mathbb{R}^n$, respectively, where $A := a_{i_{1}}..._{i_{m}}(x)y^{i_1} . . . y^{i_m}$  and $B:=c_i(x)d_j(x)y^iy^j$ with $c_id_j=c_jd_i$. Suppose that the following holds
\begin{equation}
(1+c_kd^{k})A^{ij}(B_{0j}-B_{x^j})-d^{k}\big[2\Delta_k+ (B_{0k}-B_{x^k})\big]A^{ij}c_j=0,\label{PRAsli}
\end{equation}
where
\begin{eqnarray*}
&&\Delta_k=\frac{A^{\frac{2}{m}-2} }{m}\big[(\frac{2}{m}-1)A_kA_0+AA_{0k}-AA_{x^k}\big],
\end{eqnarray*}
$d^k=g^{lk}d_l$ and $g^{lk}=[\frac{1}{2}(F^2)_{y^ly^k}]^{-1}$. Then $\bar{F}$ is projectively related to $F$. Moreover, suppose that the following holds
\begin{eqnarray}
2A^{\frac{2}{m}-2}d^ic_id^j
\big[(\frac{2}{m}-1)A_jA_0+AA_{0j}-AA_{x^j}\big]-m d^i[B_{0i}-B_{x^i}]\neq0.\label{PP}
\end{eqnarray}
Then $B=0$. In this case,  $\bar{F}=F$.
\end{thm}

\bigskip

The first to treat the conformal theory of Finsler metrics generally was  Knebelman. He defined two metric functions $F$ and $\bar{F}$ as conformal if the length of an arbitrary vector in the one is proportional to the length in the other, that is if $\bar{g_{ij}}=\varphi g_{ij}$. The length of vector $\varepsilon$ means here the fact that $\varphi g_{ij}$, as well as $g_{ij}$, must be Finsler metric tensor, he showed that $\varphi$ falls into a
point function. In this paper, we  show that if a generalized  $m$-th root  is conformal to a  $m$-th root Finsler metric, then  both of them reduce to Riemannian metrics. More precisely, we prove the following.

\begin{thm}\label{mainthm3}
Let ${\bar F}=\sqrt{A^{2/m}+B}$ and $F=A^{1/m}$ are  generalized  $m$-th root and  $m$-th root Finsler metric on an open  subset $U\subset \mathbb{R}^n$, respectively, where $A := a_{i_{1}}..._{i_{m}}(x)y^{i_1} . . . y^{i_m}$ and $B:=b_{ij}(x)y^iy^j$. Suppose that $\bar{F}$ is conformal to $F$. Then ${\bar F}$ and $F$ reduce to Riemannian metrics.
\end{thm}

%----------------------------------------------------------------
\section{Preliminaries}
%----------------------------------------------------------------------

Let $M$ be a n-dimensional $ C^\infty$ manifold. Denote by $T_x M $ the tangent space at $x \in M$, by $TM=\cup _{x \in M} T_x M $ the tangent bundle of $M$ and by $TM_{0} = TM \setminus \{ 0 \}$ the slit tangent bundle. A  Finsler metric on $M$ is a function $ F:TM \rightarrow [0,\infty)$ which has the following properties: (i) $F$ is $C^\infty$ on $TM_{0}$, (ii) $F$ is positively 1-homogeneous on the fibers of tangent bundle $TM$, (iii) for each $y\in T_xM$, the following quadratic form ${\bf g}_y$ on $T_xM$  is positive definite,
\[
{\bf g}_{y}(u,v):={1 \over 2}\frac{\partial^2}{\partial s \partial t} \left[  F^2 (y+su+tv)\right]|_{s,t=0}, \ \
u,v\in T_xM.
\]
Let  $x\in M$ and $F_x:=F|_{T_xM}$.  To measure the
non-Euclidean feature of $F_x$, define ${\bf C}_y:T_xM\otimes T_xM\otimes T_xM\rightarrow \mathbb{R}$ by
\[
{\bf C}_{y}(u,v,w):={1 \over 2} \frac{d}{dt}\left[{\bf g}_{y+tw}(u,v)
\right]|_{t=0}, \ \ u,v,w\in T_xM.
\]
The family ${\bf C}:=\{{\bf C}_y\}_{y\in TM_0}$  is called the Cartan torsion. It is well known that {\bf{C}=0} if and
only if $F$ is Riemannian.

\bigskip

Given a Finsler manifold $(M,F)$, then a global vector field ${\bf G}$ is
induced by $F$ on $TM_0$, which in a standard coordinate $(x^i,y^i)$
for $TM_0$ is given by ${\bf{G}}=y^i {{\partial} \over {\partial x^i}}-2G^i(x,y){{\partial} \over
{\partial y^i}}$, where $G^i(y)$ are local functions on $TM$ given by
\be\label{spray}
G^i:=\frac{1}{4}g^{il}\Big\{\frac{\partial^2[F^2]}{\partial x^k
\partial y^l}y^k-\frac{\partial[F^2]}{\partial x^l}\Big\},\ \
y\in T_xM.
\ee
$\bf{G}$ is called the  associated spray to $(M,F)$. The projection of an integral curve of $\bf{G}$ is called a geodesic in $M$. In local coordinates, a curve $c(t)$ is a geodesic if and only if its coordinates $(c^i(t))$ satisfy $ \ddot c^i+2G^i(\dot c)=0$.

\bigskip

For a tangent vector $y \in T_xM_0$, define ${\bf B}_y:T_xM\otimes T_xM \otimes T_xM\rightarrow T_xM$ and ${\bf E}_y:T_xM \otimes T_xM\rightarrow \mathbb{R}$ by ${\bf B}_y(u, v, w):=B^i_{\ jkl}(y)u^jv^kw^l{{\partial } \over {\partial x^i}}|_x$ and ${\bf E}_y(u,v):=E_{jk}(y)u^jv^k$, where
\[
B^i_{\ jkl}(y):={{\partial^3 G^i} \over {\partial y^j \partial y^k \partial y^l}}(y),\ \ \  E_{jk}(y):={{1}\over{2}}B^m_{\ jkm}(y),
\]
$u=u^i{{\partial } \over {\partial x^i}}|_x$, $v=v^i{{\partial } \over {\partial x^i}}|_x$ and $w=w^i{{\partial } \over {\partial x^i}}|_x$. $\bf B$ and $\bf E$ are called the Berwald curvature and mean Berwald curvature, respectively. A Finsler metric is called a Berwald metric and mean Berwald metric if $\textbf{B}=0$ or ${\bf E}=0$, respectively.

\bigskip
Define ${\bf D}_y:T_xM\otimes T_xM \otimes T_xM\rightarrow T_xM$  by
${\bf D}_y(u,v,w):=D^i_{\ jkl}(y)u^iv^jw^k\frac{\partial}{\partial x^i}|_{x}$ where
\[
D^i_{\ jkl}:=B^i_{\ jkl}-{2\over
n+1}\{E_{jk}\delta^i_l+E_{jl}\delta^i_k+E_{kl}\delta^i_j+E_{jk,l}
y^i\}.
\]
We call ${\bf D}:=\{{\bf D}_y\}_{y\in TM_{0}}$ the Douglas curvature. A Finsler metric with ${\bf D}=0$ is called a Douglas metric. It is remarkable that, the notion of Douglas metrics was proposed by B$\acute{a}$cs$\acute{o}$-Matsumoto as a generalization  of Berwald metrics \cite{BM2}.

%------------------------------------------------------------------------------------------------------------
\section{ Proof of the Theorem \ref{mainthm1}}
%------------------------------------------------------------------------------------------------------------

To prove Theorem \ref{mainthm1}, we need the following.
\begin{lem}\label{Lem1}
Let $F=\sqrt{A^{2/m}+B}$ be a generalized  $m$-th root Finsler metric on an open  subset $U\subset \mathbb{R}^n$. Then, $F$ is a locally dually flat metric if and only if the following holds
\begin{equation}
A_{x^l}=\frac{1}{2A}\Big[(\frac{2}{m}-1)A_{0}A_{l}+AA_{0l}+\frac{m}{2}A^{\frac{2m-2}{m}}(B_{0l}-2B_{x^{l}})\Big].\label{SRR0}
\end{equation}
\end{lem}
\begin{proof}
 Let F be a locally dually flat metric
\begin{equation}\label{SR0}
[F ^2]_{x^{k}y^{l}}y^k= 2[F^ 2]_{x^l}.
\end{equation}
We have
\begin{equation}\label{SR1}
(A^{\frac{2}{m}}+B)_{x^{l}}=\frac{2}{m}A^{\frac{2-m}{m}}[A_{x^{l}}+\frac{m}{2}A^{\frac{m-2}{m}}B_{x^{l}}],
\end{equation}
\begin{equation}\label{SR2}
(A^{\frac{2}{m}}+B)_{x^{k}y^{l}}y^{k}=\frac{2}{m}A^{\frac{2-m}{m}}[\frac{2-m}{m}A_{0}A_{l}A^{-1}+A_{0l}
+\frac{m}{2}A^{\frac{m-2}{m}}B_{0l}].
\end{equation}
By (\ref{SR0})-(\ref{SR2}), we have (\ref{SRR0}). The converse is trivial.
\end{proof}

\bigskip
\noindent {\it\bf Proof of Theorem \ref{mainthm1}}: Now, suppose that $A$ is irreducible. One can rewrite (\ref{SRR0}) as follows
\begin{equation}
(1-\frac{2}{m})A_{0}A_{l}-A[A_{0l}-2A_{x_{l}}]=\frac{m}{2}A^{2-\frac{2}{m}}[B_{0l}-2B_{x^{l}}].\label{SR3}
\end{equation}
The left hand side of (\ref{SR3}) is a rational function in $y$, while its right hand side is an irrational  function in $y$.  Thus, (\ref{SR3}) reduces to following
\begin{eqnarray}
&&(2-m)A_{0}A_{l}=m A[2A_{x_{l}}-A_{0l}],\label{TN}\\
&&B_{0l}-2B_{x^{l}}=0.
\end{eqnarray}
By (\ref{TN}), the irreducibility of $A$ and $deg(A_{l}) = m - 1$, it follows that there exists a 1-form $\theta = \theta_{l}y^l$  on U such that
\begin{equation}\label{SR4}
A_{0}=\theta A.
\end{equation}
This implies that
\begin{equation}\label{SR5}
A_{0l}=A\theta_l+\theta A_l-A_{x^l}.
\end{equation}
By plugging (\ref{SR4}) and (\ref{SR5}) in  (\ref{TN}), we get (\ref{SRR2}). The converse is a direct computation. This completes the proof.
\qed

\bigskip

By Lemma \ref{Lem1} and Theorem \ref{mainthm1},  we get the following.
\begin{cor} {\rm{\cite{TN1}}}
Let  $F=A^{1/m}$  be an $m$-th root Finsler metric on an open subset $U\subset \mathbb{R}^n$. Then  $F$ is a locally dually flat metric if and only if the following holds
\be
A_{x^l}=\frac{1}{2A}\Big\{(\frac{2}{m}-1)A_lA_0+AA_{0l}\Big\}.\label{d0}
\ee
Moreover, suppose that $A$ is irreducible. Then $F$ is locally dually flat if and only if there exists a 1-form $\theta=\theta_l(x)y^l$ on $U$ such that (\ref{SRR2}) holds.
\end{cor}

%------------------------------------------------------------------------------------------------------------
\section{ Proof of the Theorem \ref{mainthm2}}
%------------------------------------------------------------------------------------------------------------
Two Finsler metrics $F$ and  ${\bar F}$ on a  manifold $M$ are called projectively
related if any geodesic of the first is also geodesic for the second and vice versa. Thus,
there is a scalar function $P(x, y)$ defined on $TM_0$ such that ${\bar G}^i= G^i+Py^i$, where   $G^i$ and ${\bar G}^i$ are the geodesic spray coefficients of $F$ and  ${\bar F}$,  respectively.
\begin{lem}\label{Algebra}
Let $A=[A_{ij}]$ be an $n\times n$ invertible and symmetric matrix, $C=[C_i]$ and $D=[D_j]$ are two non-zero $n\times 1$ and $1\times n$ vector, such that $C_iD_j=C_jD_i$.   Suppose that $1+A^{pq}C_pD_q\neq 0$. Then the matrix $B=[B_{ij}]$ defined by  $B_{ij}:=A_{ij}+C_iD_j$ is invertible and
\be
B^{ij}:=(B_{ij})^{-1}=A^{ij}-\frac{1}{1+A^{pq}C_pD_q} A^{ki}A^{lj}C_kD_l,
\ee
where $A^{ij}:=(A_{ij})^{-1}$.
\end{lem}

\bigskip

\begin{lem}\label{Lem3}
Let ${\bar F}=\sqrt{A^{2/m}+B}$ and $F=A^{1/m}$ are  generalized  $m$-th root and  $m$-th root Finsler metric on an open  subset $U\subset \mathbb{R}^n$, respectively, where $A := a_{i_{1}}..._{i_{m}}(x)y^{i_1} . . . y^{i_m}$  and $B:=c_id_jy^iy^j$ with $c_id_j=c_jd_i$ and $c_id^i\neq -1$. Suppose that the following holds
\begin{equation}
mA^{\frac{m-2}{m}}A^{il}\mathfrak{B}_l-\big[4\Upsilon+k d^{l}\mathfrak{B}_l\big] \mathcal{A}^i=0,\label{PR1}
\end{equation}
where
\begin{eqnarray*}
&& \mathfrak{B}_l=B_{0l}-B_{x^l},\\
&&\Upsilon=\frac{k d^{j}}{4}\Big\{[F^{2}]_{x^ky^j}y^k-[F^{2}]_{x^j}\Big\},\\
&&k=\frac{1}{1+c_md^{m}},\\
&& \mathcal{A}^i=mA^{\frac{m-2}{m}}A^{ij}c_j.
\end{eqnarray*}
$d^l=g^{lk}d_k$ and $g^{lk}=[\frac{1}{2}(F^2)_{y^ly^k}]^{-1}$.  Then $\bar{F}$ is projectively related to $F$.
\end{lem}
\begin{proof}
By assumption, we have
\be
{\bar F}^{2}=F^{2}+B,\label{PR0}
\ee
where $F=A^{1/m}$  be an $m$-th root Finsler metric, $A:=a_{i_{1}\dots i_{m}}(x)y^{i_{1}}y^{i_{2}}\dots y^{i_{m}}$ is symmetric in all its indices and $B=c_id_j y^iy^j$. Then we have
\be
{\bar g}_{ij}=g_{ij}+c_id_j,\label{metric1}
\ee
where
\be
g_{ij}=\frac{A^{\frac{2}{m}-2}}{m^2}[mAA_{ij}+(2-m)A_iA_j].
\ee
Then by Lemma \ref{Algebra}, we get
\be
{\bar g}^{ij}=g^{ij}-\frac{1}{1+c_md^{m}}c^{i}d^{j},\label{inverse}
\ee
where $d^m=g^{ml}d_l$, $c^m=g^{ml}c_l$ and
\be
g^{ij}= A^{-\frac{2}{m}}[mAA^{ij}+\frac{m-2}{m-1}y^iy^j].\label{g}
\ee
Then by (\ref{spray}), (\ref{PR0}) and (\ref{inverse}), we have
\begin{eqnarray}
\nonumber {\bar G}^i\!\!\!\!&=&\!\!\!\!\! \frac{1}{4}{\bar g}^{il}\Big[ \frac{\partial^2{\bar F}^2}{\partial x^k \partial y^l}y^k-\frac{\partial {\bar F}^2}{\partial x^l}\Big]\\ \nonumber
\!\!\!\!&=&\!\!\!\!\! \frac{1}{4}[g^{il}-k c^{i}d^{l}]\Big[ \frac{\partial^2(F^{2}+B)}{\partial x^k \partial y^l}y^k-\frac{\partial (F^{2}+B)}{\partial x^l}\Big],
\end{eqnarray}
where $k=\frac{1}{1+c_md^{m}}$. Then
\begin{eqnarray}
\nonumber {\bar G}^i\!\!\!\!&=&\!\!\!\!\!  \frac{1}{4}[g^{il}-k c^{i}d^{l}]\Big[ \frac{\partial^2F^{2}}{\partial x^k \partial y^l}y^k-\frac{\partial F^{2}}{\partial x^l}\Big]+\frac{1}{4}[g^{il}-k c^{i}d^{l}]\Big[ \frac{\partial^2B}{\partial x^k \partial y^l}y^k-\frac{\partial B}{\partial x^l}\Big]\\
\nonumber
\!\!\!\!&=&\!\!\!\!\! G^i-\frac{k c^{i}d^{l}}{4}\Big[ \frac{\partial^2F^{2}}{\partial x^k \partial y^l}y^k-\frac{\partial F^{2}}{\partial x^l}\Big]+\frac{1}{4}[g^{il}-k c^{i}d^{l}][ B_{0l}-B_{x^l}]\\
\!\!\!\!&=&\!\!\!\!\! G^i-\big[\Upsilon+\frac{1}{4}k d^{l}\mathfrak{B}_l\big] c^i+\frac{1}{4}g^{il}\mathfrak{B}_l,\label{PR2}
\end{eqnarray}
where
\begin{eqnarray}
\Upsilon=\frac{k d^{l}}{4}\Big\{[F^{2}]_{x^ky^l}y^k-[F^{2}]_{x^l}\Big\}, \ \ \ \ \mathfrak{B}_l=B_{0l}-B_{x^l}.\label{W1}
\end{eqnarray}
Put
\begin{eqnarray}
\Phi:=\frac{m-2}{m-1}A^{-\frac{2}{m}}y^pc_p, \  \ \ \   \mathcal{A}^i:=mA^{\frac{m-2}{m}}A^{ip}c_p.\label{W2}
\end{eqnarray}
Then we have
\begin{eqnarray}
\nonumber c^i=g^{ip}c_p\!\!\!\!&=&\!\!\!\!\!A^{-\frac{2}{m}}[mAA^{ip}+\frac{m-2}{m-1}y^iy^p]c_p\\
\!\!\!\!&=&\!\!\!\!\! \mathcal{A}^i+\Phi y^i, \label{PR3}
\end{eqnarray}
By (\ref{g}), (\ref{PR2}) and   (\ref{PR3}),  we get
\begin{eqnarray}
\nonumber {\bar G}^i=G^i
\!\!\!\!&+&\!\!\!\!\ \Big[A^{-\frac{2}{m}}\frac{m-2}{4(m-1)}y^l\mathfrak{B}_l-
\big(\Upsilon+\frac{1}{4}k d^{l}\mathfrak{B}_l\big) \Phi\Big] y^i\\
\!\!\!\!&-&\!\!\!\!\ \big[\Upsilon+\frac{1}{4}k d^{l}\mathfrak{B}_l\big] \mathcal{A}^i+\frac{m}{4}A^{\frac{m-2}{m}}A^{il}\mathfrak{B}_l.\label{PR4}
\end{eqnarray}
If the relation (\ref{PR1}) holds, then by (\ref{PR4})  the Finsler metric $\bar{F}$ is projectively related to $F$.
\end{proof}

\bigskip

\begin{lem}\label{Lem4}
Let ${\bar F}=\sqrt{A^{2/m}+B}$ and $F=A^{1/m}$ are  generalized  $m$-th root and  $m$-th root Finsler metric on an open  subset $U\subset \mathbb{R}^n$, respectively, where $A := a_{i_{1}}..._{i_{m}}(x)y^{i_1} . . . y^{i_m}$ are symmetric in all its indices and $B:=c_id_jy^iy^j$ with $c_id_j=c_jd_i$ and $c_id^i\neq -1$ is a 2-form on $M$. Suppose that (\ref{PR1}) and (\ref{PP}) hold. Then $B=0$.
\end{lem}
\begin{proof}
Let (\ref{PR1}) holds
\be
\big[4\Upsilon+k d^{l}\mathfrak{B}_l\big] \mathcal{A}^i=mA^{\frac{m-2}{m}}A^{il}\mathfrak{B}_l.
\ee
Then by (\ref{W1}) and (\ref{W2}) we have
\begin{eqnarray}
g^{jl}d_j\Big[(F^{2})_{x^ky^l}y^k-(F^{2})_{x^l}+\mathfrak{B}_l\Big] A^{ip}c_p=A^{il}\mathfrak{B}_lg^{qr}c_rd_q +A^{il}\mathfrak{B}_l,\label{PR5}
\end{eqnarray}
or equivalently
\begin{eqnarray}
g^{jl}d_j\Big[(F^{2})_{x^ky^l}y^k-(F^{2})_{x^l}\Big] A^{ip}c_p-A^{il}\mathfrak{B}_l=[A^{il}d^j-A^{ij}d^l]c_j\mathfrak{B}_l.\label{Asli}
\end{eqnarray}
The following holds
\be
(F^{2})_{x^ky^l}y^k-(F^{2})_{x^l}=\frac{2}{m}A^{\frac{2}{m}-2}
\big[(\frac{2}{m}-1)A_lA_0+AA_{0l}-AA_{x^l}\big].\label{FO}
\ee
Contracting (\ref{FO}) with $g^{jl}$ yields
\be
g^{jl}[(F^{2})_{x^ky^l}y^k-(F^{2})_{x^l}]=\frac{2}{m}A^{-2}\mathbb{A}^{jl}
\big[(\frac{2}{m}-1)A_lA_0+AA_{0l}-AA_{x^l}\big],\label{F}
\ee
where $\mathbb{A}^{jl}:=[mAA^{jl}+\frac{m-2}{m-1}y^jy^l]$. By considering  (\ref{F}), the left hand side of (\ref{Asli}) is a rational function in $y$, while its right hand side is a irrational  function in $y$.  Then (\ref{Asli}) reduces to following
\begin{eqnarray}
&&g^{jl}d_j A^{ip}c_p\mathfrak{F}_l=A^{il}\mathfrak{B}_l,\label{Asli1}\\
&&A^{il}d^jc_j\mathfrak{B}_l=A^{ij}d^lc_j\mathfrak{B}_l,\label{Asli2}
\end{eqnarray}
where $\mathfrak{F}_l:=[(F^{2})_{x^ky^l}y^k-(F^{2})_{x^l}]$. Contracting (\ref{Asli1}) with $A_{si}$ implies that
\be
\mathfrak{B}_s=(d^l \ \mathfrak{F}_l)c_s.\label{Asli3}
\ee
By  multiplying (\ref{Asli2}) with $A_{is}$, we have
\be
d^jc_j\mathfrak{B}_s=d^l\mathfrak{B}_l c_s.\label{Asli4}
\ee
$d^jc_j\times$(\ref{Asli3})-(\ref{Asli4}) yields
\be
\big[(d^jc_j)d^l \mathfrak{F}_l-d^l\mathfrak{B}_l\big]c_s=0.
\ee
By assumption, (\ref{PP}) holds and  then $(d^jc_j)d^l \mathfrak{F}_l-d^l\mathfrak{B}_l\neq 0$. Thus $c_s=0$ and  $B=0$ which implies that ${\bar F}=F$. This completes the proof.
\end{proof}

\bigskip
\noindent {\it\bf Proof of Theorem \ref{mainthm2}}: By Lemmas \ref{Lem3} and \ref{Lem4}, we get the proof.
\qed

\bigskip

Recently,  Zu-Zhang-Li proved that every Douglas $m$-th root Finsler metric  $F=A^{1/m}$ ($m>4$) with irreducibility of $A$, is a Berwald metric  \cite{ZZL}. Then by Theorem \ref{mainthm2}, we have the following.
\begin{cor}
Let ${\bar F}=\sqrt{A^{2/m}+B}$ and $F=A^{1/m}$ are  generalized  $m$-th root and  $m$-th root Finsler metric on an open  subset $U\subset \mathbb{R}^n$, respectively, where $m>4$, $A := a_{i_{1}}..._{i_{m}}(x)y^{i_1} . . . y^{i_m}$ is irreducible  and $B:=c_i(x)d_j(x)y^iy^j$ with $c_id_j=c_jd_i$ and $c_id^i\neq -1$. Suppose that  (\ref{PRAsli}) holds and ${\bar F}$ is a Douglas metric. Then  $F$ reduces  to a Berwald metric.
\end{cor}

%------------------------------------------------------------------------------------------------------------
\section{ Proof of the Theorem \ref{mainthm3}}
%------------------------------------------------------------------------------------------------------------
Let $(M,F)$ and $(M,\bar{F})$ be two Finsler spaces on same underlying $n$-dimensional manifold $M$. A Finsler space $(M, F)$ is conformal to a Finsler space $(M,\bar{F})$, if and only if there exists a scalar  field $\alpha(x)$ satisfying  $\bar{F}=e^{\alpha} F$ (see \cite{Ha}). The conformal change $\alpha(x)$ is called homothetic and isometry if $\alpha_i=\frac{\partial\alpha}{\partial x^i}=0$ and $\alpha=0$, respectively. In these section, we will prove a generalized version of Theorem  \ref{mainthm3}. Indeed,  we  are going to consider two generalized $m$-th root metrics ${\bar F}=\sqrt{A^{2/m}+B}$ and $\tilde{F}=\sqrt{A^{2/m}+\tilde{B}}$ which are conformal and prove the following.

\begin{thm}\label{mainthm4}
Let ${\bar F}=\sqrt{A^{2/m}+\bar{B}}$ and $\tilde{F}=\sqrt{A^{2/m}+\tilde{B}}$ are two generalized  $m$-th root metrics on an open  subset $U\subset \mathbb{R}^n$, where $\bar{B}:=\bar{b}_{ij}(x)y^iy^j$ and $\tilde{B}:=\tilde{b}_{ij}(x)y^iy^j$. Suppose that $\bar{F}$ is non-isometry  conformal to $\tilde{F}$. Then $F=A^{1/m}$ is a Riemannian metric.
\end{thm}
\begin{proof}
Let
\be
{\bar F}=e^\alpha \tilde{F},\label{m1}
\ee
where ${\bar F}=\sqrt{A^{2/m}+B}$ and $\tilde{F}=\sqrt{A^{2/m}+\tilde{B}}$  are generalized $m$-th root Finsler metrics on an open  subset $U\subset \mathbb{R}^n$, where $\bar{B}:=\bar{b}_{ij}(x)y^iy^j$ and $\tilde{B}:=\tilde{b}_{ij}(x)y^iy^j$. By assumption ${\bar F}$ is conformal to $\tilde{F}$. Then, we have
\be
{\bar g}_{ij}=e^{2\alpha} \tilde{g}_{ij}.\label{m2}
\ee
Then we have
\be
g_{ij}+\bar{b}_{ij}=e^{2\alpha}(g_{ij}+\tilde{b}_{ij}),\label{m3}
\ee
where $g_{ij}=\frac{1}{2}(A^{\frac{2}{m}})_{y^iy^j}$ is the fundamental tensor of $F:=A^{1/m}$. Since $\alpha$ is not isometry, i.e., $\alpha\neq 0$, then by (\ref{m2}) and (\ref{m3}), we get
\be
g_{ij}=\frac{1}{1-e^{2\alpha}}(e^{2\alpha}\tilde{b}_{ij}-\bar{b}_{ij}).\label{m4}
\ee
This implies that $C_{ijk}=0$ and then $F$ is Riemannian.
\end{proof}

By (\ref{m3}), we get the following.
\begin{cor}
Let ${\bar F}=\sqrt{A^{2/m}+\bar{B}}$ and $\tilde{F}=\sqrt{A^{2/m}+\tilde{B}}$ are two  generalized  $m$-th root metrics on an open  subset $U\subset \mathbb{R}^n$, where $F:=A^{1/m}$ is not Riemannian,  $\bar{B}:=\bar{b}_{ij}(x)y^iy^j$ and $\tilde{B}:=\tilde{b}_{ij}(x)y^iy^j$. Suppose that $\bar{F}$ is conformal to $\tilde{F}$. Then $\bar{F}=\tilde{F}$ or equivalently $\bar{B}=\tilde{B}$.
\end{cor}

\bigskip
\noindent {\it\bf Proof of Theorem \ref{mainthm3}}:  In Theorem \ref{mainthm4}, put $\tilde{B}=0$ and $\tilde{F}:=F$. Suppose that  the generalized $m$-th root metric ${\bar F}=\sqrt{A^{2/m}+B}$ is conformal to the $m$-th root Finsler metric $F=A^{1/m}$. By Theorem \ref{mainthm4}, $F$ is Riemannian and then $C_{ijk}=0$. Since ${\bar g}_{ij}=e^{2\alpha} g_{ij}$ then ${\bar g}_{ij}=g_{ij}+b_{ij}$, which yields
\[
\bar{C}_{ijk}=C_{ijk}.
\]
Thus $\bar{C}_{ijk}=0$, which implies that  $\bar{F}$ reduces to a Riemannian metric.  This completes the proof.
\qed

\bigskip
\noindent
{\bf Acknowledgments.}  We are deeply grateful to the referee for a very careful reading of the manuscript and valuable suggestions.

Akbar Tayebi and Mohammad Shahbazi \\
Department of Mathematics, Faculty  of Science\\
University of Qom \\
Qom. Iran\\
Email:\ akbar.tayebi@gmail.com\\
Email:\ m.shahbazinia@gmail.com
\bigskip

\noindent
Esmaeil Peyghan\\
Department of Mathematics, Faculty  of Science\\
Arak University\\
Arak 38156-8-8349,  Iran\\
Email: epeyghan@gmail.com


\begin{thebibliography}{MaHo}
%-----------------------------------------------------------------------------------------------
\bibitem{am} S.-I. Amari, {\it Differential-Geometrical Methods in Statistics}, Springer Lecture Notes in Statistics,  Springer-Verlag, 1985.
%-----------------------------------------------------------------------------------------------
\bibitem{amna} S.-I. Amari and H. Nagaoka, {\it Methods of Information Geometry}, AMS Translation of Math. Monographs,  Oxford University Press, 2000.
%-----------------------------------------------------------------------------------------------
\bibitem{AIM} P. L. Antonelli, R. Ingarden and M. Matsumoto, {\it The Theory of Sprays and Finsler
Spaces with Applications in Physics and Biology}, Kluwer Acad. publ., Netherlands, 1993.
%-----------------------------------------------------------------------------------------------
\bibitem{Asan} G.S. Asanov, {\it Finslerian Extension of General Relativity}, Reidel, Dordrecht, 1984.
%---------------------------------------------------------------
 \bibitem{BM2} S.  B\'{a}cs\'{o} and M. Matsumoto, {\it On Finsler spaces of Douglas type, A generalization of notion of Berwald space}, Publ. Math. Debrecen. {\bf 51}(1997), 385-406.
%----------------------------------------------------------------------------------------------
\bibitem{B} V. Balan, {\it Spectra of symmetric tensors and $m$-root Finsler models}, Linear Algebra and its Applications,  {\bf 436}(1) (2012), 152-162.
%----------------------------------------------------------------------------------------------
\bibitem{Balan} V. Balan, {\it  Notable submanifolds in Berwald-Mo\'{o}r spaces}, BSG Proc. 17, Geometry Balkan Press 2010, 21-30.
%----------------------------------------------------------------------------------------------
\bibitem{Balan2} V. Balan, {\it CMC and minimal surfaces  in  Berwald-Mo\'{o}r spaces}, Hypercomplex Numbers in Geometry and Physics, 2(6), {\bf 3}(2006), 113-122.
%----------------------------------------------------------------------------------------------
\bibitem{BBL} V. Balan,  N. Brinzei and S. Lebedev, {\it Geodesics, paths and Jacobi fields for Berwald-Mo\'{o}r quartic metrics}, Hypercomplex Numbers in Geometry and Physics, accepted.
%----------------------------------------------------------------------------------------------
\bibitem{Mangalia} V. Balan and N. Brinzei, {\it Einstein equations for $(h,v)$-Berwald-Mo\'{o}r relativistic models}, Balkan. J. Geom. Appl. {\bf 11}(2) (2006), 20-27.
%-----------------------------------------------------------------------------------------------
\bibitem{Cairo1} V.  Balan and N. Brinzei, {\it Berwald-Mo\'{o}r-type $(h,v)$-metric physical models}, Hypercomplex Numbers in Geometry and Physics. {\bf 2}(4) (2005), 114-122.
%----------------------------------------------------------------------------------------------
\bibitem{Ha} M. Hashiguchi, {\it On conformal transformation of Finsler metrics},  J. Math. Kyoto Univ.  {\bf 16}(1976), 25-50.
%----------------------------------------------------------------------------------------------
\bibitem{Leb} S.V. Lebedev, {\it The generalized Finslerian metric tensors}, to appear.
%-----------------------------------------------------------------------------------------------
\bibitem{shLi1} B. Li and Z. Shen, {\it On projectively flat fourth root metrics}, Canad. Math. Bull. {\bf 55}(2012), 138-145.
%-----------------------------------------------------------------------------------
\bibitem{MatShim} M. Matsumoto and H. Shimada, {\it On Finsler spaces with 1-form metric. II. Berwald-Mo\'{o}r's metric $L=\left( y^{1}y^{2}...y^{n}\right) ^{1/n}$}, Tensor N. S. {\bf 32}(1978), 275-278.
%-----------------------------------------------------------------------------------------------
\bibitem{Pav1} D.G. Pavlov, {\it Space-Time Structure, Algebra and Geometry}, Collected papers, TETRU, 2006.
%-----------------------------------------------------------------------------------------------
\bibitem{Pav2} D.G. Pavlov, {\it Four-dimensional time}, Hypercomplex Numbers in Geometry and Physics, {\bf 1}(2004), 31-39.
%-----------------------------------------------------------------------------------------------
\bibitem{shen1} Z. Shen, {\it Riemann-Finsler geometry with applications to information geometry}, Chin. Ann. Math. {\bf 27}(2006), 73-94.
%-----------------------------------------------------------------------------------------------
\bibitem{Shim} H. Shimada, {\it On Finsler spaces with metric }$L=%
\sqrt[m]{a_{i_{1}i_{2}...i_{m}}y^{i_{1}}y^{i_{2}}...y^{i_{m}}},$ Tensor, N.S., {\bf 33}(1979), 365-372.
%----------------------------------------------------------------------------------------------
\bibitem{Tamassy} L. Tam\`{a}ssy, {\it Finsler spaces with polynomial metric}, Hypercomplex Numbers in Geometry and Physics, {\bf 3}(6) (2006), 85-92.
%-----------------------------------------------------------------------------------------------
\bibitem{TN1} A. Tayebi and B. Najafi, {\it  On $m$-th root Finsler metrics}, J. Geom. Phys. {\bf 61}(2011), 1479-1484.
%-----------------------------------------------------------------------------------------------
\bibitem{TN2} A. Tayebi and B. Najafi, {\it   On $m$-th root metrics with special curvature properties}, C. R. Acad. Sci. Paris, Ser. I, {\bf 349}(2011),  691-693.
%-----------------------------------------------------------------------------------------------
\bibitem{YY} Y. Yu and Y. You, {\it On Einstein $m$-th root metrics}, Diff. Geom. Appl. {\bf 28}(2010) 290-294.
%-----------------------------------------------------------------------------------------------
\bibitem{ZZL} D. Zu, S. Zhang and B. Li, {\it On Berwald $m$-th root Finsler metrics}, Publ. Math .Debrecen,  accepted.

%-----------------------------------------------------------------------------------------------
\end{thebibliography}
\end{document}